\documentclass{article}
\usepackage{amsmath}
\usepackage{amssymb}
\usepackage{amsthm}
\usepackage{graphicx}
\usepackage{enumerate}
\usepackage[colorlinks=true,urlcolor=blue,linkcolor=blue,plainpages=false,pdfpagelabels]{hyperref}

\newtheorem{theorem}{Theorem}[section]
\newtheorem{proposition}[theorem]{Proposition}
\newtheorem{lemma}[theorem]{Lemma}

\newcommand{\beq}{\begin{equation}}
\newcommand{\eeq}{\end{equation}}
\newcommand{\be}{\begin{enumerate}}
\newcommand{\ee}{\end{enumerate}}
\newcommand {\bua} {\begin{eqnarray*}}
\newcommand {\eua} {\end {eqnarray*}}

\newcommand{\details}[1]{}

\title{Effective strong convergence of the proximal point algorithm in CAT(0) spaces}
\author{Lauren\c tiu Leu\c stean${}^{a,b,d}$ and Andrei Sipo\c s${}^{c,d}$\\[2mm]
\footnotesize ${}^{a}$The Research Institute of the University of Bucharest (ICUB), University of Bucharest,\\
\footnotesize Bd. M. Kog\u{a}lniceanu 36-46, 050107, Bucharest, Romania\\[1mm]
\footnotesize ${}^{b}$Faculty of Mathematics and Computer Science, University of Bucharest,\\
\footnotesize  Academiei 14, 010014, Bucharest, Romania\\[1mm]
\footnotesize ${}^c$Department of Mathematics, Technische Universit\"at Darmstadt,\\
\footnotesize Schlossgartenstrasse 7, 64289 Darmstadt, Germany\\[1mm]
\footnotesize ${}^d$Simion Stoilow Institute of Mathematics of the Romanian Academy,\\
\footnotesize Calea Grivi\c tei 21, 010702 Bucharest, Romania \\[2mm]
\footnotesize E-mails: laurentiu.leustean@unibuc.ro, sipos@mathematik.tu-darmstadt.de\\
}
\date{}

\begin{document}

\maketitle

\begin{center}
{\em  Dedicated to Professor Simeon Reich on the occasion of his 70th birthday}
\end{center}

\begin{abstract}
\noindent We apply methods of proof mining to obtain uniform quantitative bounds on the strong convergence  of the proximal point algorithm  for finding minimizers of convex, lower 
semicontinuous proper functions in CAT(0) spaces. Thus, for uniformly convex functions we compute rates of convergence, while, for totally bounded CAT(0) spaces we apply methods 
introduced by Kohlenbach, the first author and Nicolae to compute rates of metastability. \\


\end{abstract}

\section{Introduction}

Let $X$ be a complete CAT(0) space and $f: X \to (-\infty, +\infty]$ be a convex, lower 
semicontinuous (lsc), proper function that has at least one minimizer. In analogy to the 
celebrated notion in Hilbert spaces, Jost introduced \cite{Jos95} in 1995 the {\it resolvent} 
$J_\gamma$ of $f$ of order $\gamma>0$ as
\[J_\gamma:X \to X, \quad J_\gamma(x) = {\arg\!\min}_{y \in X} \left[f(y) + \frac1{2\gamma} d^2(x,y)\right].\]
It may be shown that when $f$ satisfies the conditions specified above, the $\arg\!\min$ of 
the right hand side is well-defined, i.e. it exists and is unique.

Resolvent operators corresponding to diverse forms of objects had already been defined and applied 
to great success in the field of convex optimization in linear settings. The main tool used to exploit them, 
developed by Martinet \cite{Mar70}, Rockafellar \cite{Roc76} and Br\'ezis and Lions \cite{BreLio78}, 
was the {\it proximal point algorithm}, and for a general introduction to this field in the 
context of Hilbert spaces, putting this algorithm in its context, one may see the book of 
Bauschke and Combettes \cite{BauCom17}.

This algorithm was first proven to work in the setting of complete CAT(0) spaces by Ba\v{c}\'ak 
\cite{Bac13}, who showed some years ago its weak convergence (that is, $\Delta$-convergence). 
Specifically, in the above context, if one takes a starting point $x \in X$, and a sequence 
of step-sizes $(\gamma_n)_{n \in \mathbb{N}}$, then one may define the proximal point algorithm  
$(x_n)_{n \in \mathbb{N}}$ by setting:
\begin{center} $x_0:=x$, \quad  $x_{n+1} := J_{\gamma_n} x_n$  for any $n\in\mathbb{N}$. 
\end{center}
It is of this sequence that Ba\v{c}\'ak proved the weak convergence to a minimizer of $f$, 
assuming additionally that the series $\sum_{n=0}^\infty \gamma_n$ 
is divergent.

It is interesting to ask whether the algorithm may converge in the usual sense (i.e., strongly). 
Although this is known not to hold in general (see \cite[Corollary 5.1]{Gul91} for a 
counterexample in the context of Hilbert spaces), Ba\v{c}\'ak selected two cases where it 
can be guaranteed to happen. This is expressed by the following result (cf. \cite[Remark 1.7]{Bac13}).

\begin{theorem}\label{ppa-cat0-strong}
In the above hypotheses, assume further either that $f$ is uniformly convex or that $X$ is locally compact. 
Then ${(x_n)}$ converges strongly to a minimizer of $f$. 
\end{theorem}

Our goal in this paper will be to prove a quantitative version of the above Theorem 
\ref{ppa-cat0-strong}, as seen through the lens of proof mining.

Proof mining is an applied subfield of mathematical logic (see \cite{Koh08} for a 
comprehensive introduction and \cite{Koh17,Koh18} for a survey of recent results), whose aim is 
to obtain quantitative information -- primarily witnesses and bounds for existentially quantified 
variables -- out of ordinary mathematical proofs that cannot be said to be fully constructive. 
For example, in the case above, where the conclusion of the proof is a convergence statement, 
the relevant piece of quantitative information would be a formula -- that is, a function of 
low computational complexity -- yielding, for a given $\varepsilon$, the rank from which all 
terms of the sequence are closer than $\varepsilon$ to the desired limit. (In this case, a bound 
will also be a witness, by the monotonicity of the statement.) Unfortunately, fundamental 
results in computability theory lead to the impossibility of such a formula -- called a {\it rate of convergence} -- 
in relatively simple contexts (e.g. ergodic averages), and therefore the most general metatheorems 
of proof mining prohibit it from the start by restricting the logical complexity of the formulas 
that may serve as a conclusion to a proof which is under analysis, if that proof makes use of the 
law of excluded middle. A significant part of the recent research in the field has been devoted 
to finding methods of `circumventing' this fundamental limitation.

The two cases of the theorem under discussion shall be analyzed separately, as one deals 
with the problem above in different ways. The case where $f$ is uniformly convex is a 
particularly interesting one. It is known that this property immediately yields the uniqueness 
of the minimizer, and past work of Kohlenbach \cite{Koh90} and Briseid \cite{Bri09} has shown 
that in this case one may obtain convergence of the sequence in a completely constructive way, 
and hence extract a full rate of convergence out of that argument. Here, we reason more directly, 
in the sense that we can simply get out of Ba\v{c}\'ak's own argument the needed rate of convergence 
for the algorithm in this `uniform' case. This is what we do in Section~\ref{conv-unif}. 
We mention that such a rate for a quite general uniformity 
condition (though not including the case examined here) was previously obtained in the paper \cite{LeuNicSip17}.

In the second case, when $X$ is locally compact, things are not this simple, as excluded middle features prominently in 
the proof and one cannot so easily eliminate it. One then works with a related notion of {\it metastability} -- 
introduced by Tao in \cite{Tao07} and used extensively in his proof of the convergence 
theorem for multiple ergodic averages \cite{Tao08} -- a property equivalent to that of a 
sequence being Cauchy. This property is expressed as follows:
$$\forall k \in {\mathbb N} \,\forall g: {\mathbb N} \to {\mathbb N} \,\exists N \in {\mathbb N} 
\,\forall i,j \in [N, N+g(N)]\,\,\, \left(d(x_i,x_j) \leq \frac1{k+1}\right),$$
where $[N,N+g(N)]$ denotes the set $\{N,N+1,\ldots, N+g(N)\}$.
By examining the above form in conjunction with the logical metatheorems of proof mining, 
it can be seen that it fits into the required logical complexity and therefore one can extract 
from a classical proof of the statement a {\it rate of metastability}, i.e. a functional  
$\Phi:\mathbb{N}\times\mathbb{N}^\mathbb{N}\to \mathbb{N}$ satisfying 
\begin{equation}\label{def-rate-meta}
\exists N \leq \Phi(k,g)\,\forall i,j \in [N, N+g(N)]\,\,\, \left(d(x_i,x_j) \leq \frac1{k+1}\right),
\end{equation}
for all $k\in\mathbb{N}$ and all $g: {\mathbb N} \to {\mathbb N}$. Such rates, even if they are not 
full rates of convergence, may 
still be deemed useful, for example in the further analysis of proofs of well-behaved statements 
which nevertheless use the convergence result as a premise.

Now, the proof of the second case of Theorem~\ref{ppa-cat0-strong} consists essentially of an argument that crucially 
uses the property of the sequence generated by the algorithm being {\it Fej\'er monotone} with 
respect to the set of minimizers of the function. Such an argument has been analyzed before in a 
highly general form by Kohlenbach, the first author and Nicolae \cite{KohLeuNic18}, an analysis 
that has been since applied to other similar instances \cite{KohLopNic17,LeuRadSip16}, therefore 
illustrating how modular the techniques of proof mining can become. An application 
of this idea to the proximal point algorithm in CAT(0) spaces has been carried out in \cite{LeuSip17}. 
The goal of our analysis in Section~\ref{meta-tot-bound} will be to show that one may use 
a way of approximating the minimizer  set that captures more accurately the spirit of the problem in order 
to obtain, for complete totally bounded CAT(0) spaces, a rate of metastability 
for the sequence together with a precise finitization of the fact that the limit  of 
the sequence is a minimizer of the function.

We finish the introduction by recalling some notions that are used for expressing our 
quantitative results. We denote $[m,n]=\{m,m+1,\ldots,n\}$ for all $m, n\in\mathbb{N}$ with $m\leq n$. For any 
$f:\mathbb{N} \to \mathbb{N}$, we define 
$$f^M: \mathbb{N} \to \mathbb{N}, \quad f^M(n) = \max_{i \leq n} f(i).$$
Then $f^M$ is nondecreasing and  $f(n) \leq f^M(n)$ for all $n\in\mathbb{N}$.

Let $(X,d)$ be a metric space and $(a_n)$ be a convergent sequence in $X$ with $\displaystyle\lim_{n\to\infty} a_n=a$.  
A mapping $\varphi:{\mathbb N}\to{\mathbb N}$ is said to be:
\begin{enumerate}
\item[(i)] a {\it rate of convergence} of $(a_n)$ if for all $k\in\mathbb{N}$,
\[
d(a_n,a)\leq \frac1{k+1} \quad \text{ for all~}n\geq \varphi(k).
\]
\item[(ii)] a {\it Cauchy modulus} of $(a_n)$ if  for all $k\in\mathbb{N}$,
\[
d(a_m,a_n)\leq \frac1{k+1} \quad \text{ for all~}m,n\geq \varphi(k).
\]
\end{enumerate}
If $(b_n)$ is a sequence of nonnegative reals such that the series $\sum\limits_{n=0}^\infty b_n$ 
diverges, then a {\it rate of divergence}  of $\sum\limits_{n=0}^\infty b_n$ is a mapping
$\theta:{\mathbb N}\to{\mathbb N}$  satisfying
\[
\sum_{n=0}^{\theta(P)}b_n \geq P \quad \text{ for all~}P \in {\mathbb N}.
\]

\section{Preliminaries and first results}

We briefly recall some notions on geodesic spaces.  Let $(X,d)$ be a metric space. A geodesic 
in $X$ is a map $\gamma:[a,b]\to X$ (where $a,b \in \mathbb{R}$) 
satisfying 
\[
d(\gamma(s),\gamma(t))=|s-t| \quad \text{~~for all~~} s,t\in [a,b].
\]
A geodesic segment in $X$ is the image $\gamma([a,b])$ of a geodesic $\gamma:[a,b]\to X$. 
If we denote  $x:=\gamma(a)$ and $y:=\gamma(b)$, then we say that the geodesic  segment 
$\gamma([a,b])$ joins $x$ and $y$\\.  The metric space $(X,d)$ is a (uniquely) geodesic space 
if every two distinct points are joined by a  (unique) geodesic segment. One can easily see that 
$(X,d)$ is a uniquely geodesic space if and only if  for any $x,y\in X$ there exists  a  unique
 geodesic $\gamma : [0,d(x,y)] \to X$ such that $\gamma(0)=x$ and $\gamma(d(x,y))=y$. 

A CAT(0) space is a geodesic space $(X,d)$ satisfying the following:  for all $z \in X$, all geodesics 
$\gamma : [a,b] \to X$ and all $t \in [0,1]$, 
\begin{equation}
d^2(z,\gamma((1-t)a+tb)) \leq (1-t)d^2(z,\gamma(a)) + td^2(z,\gamma(b)) - t(1-t)d^2(\gamma(a),\gamma(b)). \label{def-CAT0-0}
\end{equation}
It is a well-known fact that CAT(0) spaces are uniquely geodesic. 

\mbox{}

In the sequel, $(X,d)$  is a CAT(0) space.  If $x,y\in X$ and
 $\gamma : [0,d(x,y)] \to X$ is the unique geodesic that joins them, we denote, for any $t \in [0,1]$, 
the point $\gamma(td(x,y))$ by $(1-t)x + ty$. Thus,
\[d(x,(1-t)x + ty)=td(x,y) \quad\text{and}\quad d(y,(1-t)x + ty)=(1-t)d(x,y).\]

Using this notation, it is an immediate exercise to see that \eqref{def-CAT0-0} can be written as 
follows: for all $x,y,z\in X$ and all $t\in [0,1]$, 
\begin{equation}
d^2(z,(1-t)x+ty)) \leq (1-t)d^2(z,x) + td^2(z,y) - t(1-t)d^2(x,y). \label{def-CAT0}
\end{equation}

Let $f: X \to (-\infty, +\infty]$ be a convex, lower semicontinuous (lsc), 
proper function.  We denote by $Argmin(f)$  the set of minimizers of $f$  and we assume that $Argmin(f)$ is nonempty.  
For any $\gamma>0$, the {\em resolvent} $J_\gamma$ of $f$ of order $\gamma>0$ \cite{Jos95} is defined as follows:
\[J_\gamma:X \to X, \quad J_\gamma(x) = {\arg\!\min}_{y \in X} \left[f(y) + \frac1{2\gamma} d^2(x,y)\right].\]
We also say that $J_\gamma$ is the {\it proximal mapping} of order $\gamma$. As proved in \cite{Jos95},
$J_\gamma$ is a nonexpansive mapping: for all $x,y\in X$, 
\[d(J_{\gamma}x,J_{\gamma}y)\leq d(x,y).\]

If $T:X\to X$, we denote by $Fix(T)$ the set of fixed points of $T$. The definition of the proximal mapping is motivated by the following well-known result. 

\begin{proposition}\label{fix-min}
For any $\gamma > 0$, $Fix(J_{\gamma}) =Argmin(f)$.
\end{proposition}

Let  $(\gamma_n)$  be a sequence in $ (0, \infty)$. The {\it proximal point algorithm} 
$(x_n)$ starting with $x\in X$ is defined as follows:
\begin{equation}
x_0 :=x, \qquad x_{n+1} := J_{\gamma_n} x_n \text{~for all~} n\in\mathbb{N}. \label{def-PPA}
\end{equation}
It is easy to see  that for all $n,l\in \mathbb{N}$ and all $q \in X$,
\begin{eqnarray}
d(x_{n+l},q) &\leq&  d(x_n, q) + \sum_{i=n}^{n+l-1} d(q, J_{\gamma_i}q) \label{dxnl-p}.
\end{eqnarray}

As a consequence, we get that 

\begin{lemma}\label{xn-Fejer}
$(x_n)$ is Fej\'er monotone with respect to $Argmin(f)$, that is
\[
d(x_{n+1},p) \leq   d(x_n, p)
\]
for all $p\in Argmin(f)$ and all $n\in \mathbb{N}$. 
\end{lemma}

We recall the following result that will be used in this paper.

\begin{proposition}\cite[Lemma 2.6.(ii)]{LeuSip17}\label{prop-limfxn-minf}$\,$\\
Let $b\in \mathbb{R}$  be such that $d(x,p)\leq b$ for some $p\in Argmin(f)$. Assume that 
$\sum_{n=0}^\infty \gamma_n = \infty$ with rate of divergence $\theta$. 

Then $\displaystyle\lim_{n \to \infty} f(x_n) = \min(f)$, with a (nondecreasing) rate of convergence
\begin{equation}
\beta_{b, \theta}:\mathbb{N}\to\mathbb{N}, \quad \beta_{b, \theta}(k)=\theta^M\left(\left\lceil \frac{b^2(k+1)}2 \right\rceil \right) +1.\label{rate-fxn}
\end{equation}
\end{proposition}

We define, 
for any $k \in \mathbb{N}$ and any $\gamma>0$,
\[
Argmin_k(\gamma f) := \left\{x \in X \bigm\vert \text{for all }y \in X\text{, } \gamma f(x) 
\leq \gamma f(y) + \frac 1 {k+1}\right\}. 
\]

We call an element of $Argmin_k(\gamma f)$ a $k$-{\it approximate minimizer} of $\gamma f$. In addition, we define 
a $k$-{\it approximate fixed point} of an operator $T:X\to X$ to be a point $x$ such that
$$d(x,Tx) \leq \frac1{k+1}.$$

The following result will be  needed  in Section~\ref{meta-tot-bound}.

\begin{proposition}\label{l-l}
Define
\begin{equation}
\psi:\mathbb{N}\to \mathbb{N}, \quad  \psi(k)=2(k+1)^2 -1. \label{appmin-def}
\end{equation}

Then, for all  $\gamma >0$ and all $k \in \mathbb{N}$, any $\psi(k)$-approximate 
minimizer of $\gamma f$ is a $k$-approximate fixed point of $J_\gamma$.
\end{proposition}
\begin{proof}
Let $x$ be a $\psi(k)$-approximate minimizer of $\gamma f$. Then we have that
$$ \gamma f(x) \leq \gamma f(J_\gamma x) + \frac1{2(k+1)^2}.$$
On the other hand, by the definition of $J_\gamma$,
$$ \gamma f(J_\gamma x) + \frac12 d^2(x,J_\gamma x) \leq \gamma f(x).$$
Putting these together, we immediately obtain that
$$d(x,J_\gamma x) \leq \frac1{k+1}.$$
\end{proof}

\section{ A rate of convergence for uniformly convex functions}\label{conv-unif}

Let us recall that, if $(X,d)$ is a metric space, then  a function  $f:X\to (-\infty,\infty]$ is  {\it uniformly convex} if there exists a nondecreasing function 
$\varphi : [0, \infty) \to [0, \infty]$  vanishing only at $0$ such that  for all $x,y \in X$ and 
all $t \in [0,1]$,
\begin{equation}
f((1-t)x + ty) \leq (1-t) f(x) + tf(y) - t(1-t)\varphi(d(x,y)). \label{uconv}
\end{equation}
Such a function $\varphi$ is called a {\it modulus of uniform convexity} of $f$. We also 
say that $f$ is uniformly convex with modulus $\varphi$.

The main result of this section is the following effective version of Theorem~\ref{ppa-cat0-strong} for 
uniformly convex functions. 

\begin{theorem}\label{quant-thm-uc}
Let $X$ be a complete CAT(0) space and $f : X \to (-\infty,\infty]$ be an  lsc proper function with $Argmin(f)\ne\emptyset$.  Assume, furthermore, that $f$ is 
uniformly convex with modulus $\varphi$. 

Suppose that $x\in X$  and $b>0$ are such that $d(x,p)\leq b$ for some minimizer $p$ of $f$ 
and that $(\gamma_n)$  is a sequence in $(0,\infty)$ such that $\sum_{n=0}^\infty \gamma_n = \infty$ with 
rate of divergence $\theta$.

Let  $(x_n)$ be the proximal point algorithm starting with $x$.
Define
$$\Sigma_{b,\theta,\varphi}:\mathbb{N}\to\mathbb{N}, \quad \Sigma_{b,\theta,\varphi}(k)= \beta_{b,\theta}\left(\left\lceil\frac{8}{\varphi
\left(\frac1{k+1}\right)}\right\rceil \right), $$
with $\beta_{b, \theta}$ given by \eqref{rate-fxn}.

Then $\Sigma_{b,\theta,\varphi}$ is a rate of convergence for  $(x_n)$ and the limit of this sequence is a minimizer of $f$.
\end{theorem}

\begin{proof}
Let $k \in \mathbb{N}$ and $n,m \geq \Sigma_{b,\theta,\varphi}(k)$ be arbitrary. Applying \eqref{uconv} for $x_n$, $x_m$ and $\frac12$, 
we get that
\begin{equation*}
f\left(\frac{x_n+x_m}2\right) \leq  \frac12 f(x_n) + \frac12 f(x_m) - \frac14 \varphi(d(x_n,x_m)),
\end{equation*}
so 
\begin{align*}
\frac14 \varphi(d(x_n,x_m)) \leq &\ \frac12 f(x_n) + \frac12 f(x_m) - f\left(\frac{x_n+x_m}2\right) \\
\leq &\ \frac12 f(x_n) + \frac12 f(x_m) - \min(f).
\end{align*}
By Proposition~\ref{prop-limfxn-minf}, we have that $\displaystyle\lim_{n \to \infty} f(x_n) = \min(f)$, 
with rate of convergence $\beta_{b, \theta}$.   Using the definition of $\Sigma_{b,\theta,\varphi}$, we get that 
\[f(x_n) - \min(f), f(x_m) - \min(f)\leq \frac1{\left\lceil\frac{8}{\varphi\left(\frac1{k+1}\right)}\right\rceil + 1}.\]
It follows that 
\begin{align*}
\frac14 \varphi(d(x_n,x_m)) \leq &\ \frac1{\left\lceil\frac{8}{\varphi\left(\frac1{k+1}\right)}\right\rceil + 1} \leq 
\frac18\varphi\left(\frac1{k+1}\right)<\frac14\varphi\left(\frac1{k+1}\right).
\end{align*}
Since $\varphi$ is nondecreasing, we must have 
$$d(x_n,x_m) \leq \frac1{k+1}.$$
Thus, we have proved that $(x_n)$ is Cauchy with  $\Sigma_{b,\theta,\varphi}$ being a Cauchy modulus. 
Since $X$ is complete, $(x_n)$ is also convergent and let $z$ be its limit. Setting $m \to \infty$ in the above relation, we obtain that
$$d(x_n, z) \leq \frac1{k+1},$$
proving that $\Sigma_{b,\theta,\varphi}$ is indeed a rate of convergence for $(x_n)$.

Finally, by \cite[Remark 1.7]{Bac13} (see Theorem~\ref{ppa-cat0-strong}), the limit $z$ of $(x_n)$ is a minimizer of $f$.
\end{proof}

\section{A rate of metastability for totally bounded spaces}\label{meta-tot-bound}

Let $(X,d)$ be a metric space and $F\subseteq X$. The following notion, introduced in 
\cite{KohLeuNic18}, is essential for this section: an {\it approximation} of  $F$ is a  
family $(AF_k)_{k \in \mathbb{N}}$ of subsets of $X$ 
satisfying 
$$F = \bigcap_{k \in \mathbb{N}} AF_k \quad \text{ and } \quad AF_{k+1} \subseteq AF_k \text{~for all~} k\in\mathbb{N}.$$
Using such approximations, Kohlenbach, the first author and Nicolae developed in \cite{KohLeuNic18} 
very general methods that can be used to obtain quantitative forms, providing uniform rates of metastability,
of strong convergence results in compact spaces for sequences $(x_n)$ satisfying a general form 
of Fej\'er monotonicity. These methods have been already applied to the proximal point 
algorithm: in \cite{KohLeuNic18} with $F$ being the set of zeros of a maximally monotone operator 
in a Hilbert space and in 
\cite{LeuSip17} for the setting from this paper. The approximation used in \cite{LeuSip17} for 
$F:=Argmin(f)$ was defined as:
\[
AF_k := \left\{x \in X \mid \text{for all }i \leq k\text{, }d(x,J_{\gamma_i}x) \leq \frac1{k+1} \right\}.
\]
In this paper, we  use the following approximation of $F$, which most faithfully exhibits the meaning of the problem: 
\begin{align*}
AF_k := Argmin_k(f) =&\, \left\{x \in X \bigm\vert \text{for all }y \in X\text{, } 
f(x) \leq f(y) + \frac 1 {k+1}\right\}.
\end{align*}
Obviously, $Argmin_{k+1}(f)\subseteq Argmin_k(f)$  for every $k\in \mathbb{N}$ and 
$Argmin(f)=\bigcap_{k \in \mathbb{N}} Argmin_k(f)$.

We recall \cite{Ger08,KohLeuNic18} that a {\it modulus of total boundedness} for  a metric space 
$(X,d)$ is a function $\alpha : \mathbb{N} \to \mathbb{N}$ such that for any $k \in \mathbb{N}$ 
and any sequence ${(x_n)}$ in $X$,
\begin{center} there exist $i,j\in [0,\alpha(k)], i<j$ such that $d(x_i,x_j) \leq \frac1{k+1}$.
\end{center}
It is an easy exercise to verify that $X$ is totally bounded if and only if it has a modulus  of total boundedness.
Furthermore, it is well-known that a metric space is compact if and only if it is complete and totally bounded.

The main result of this section is an effective version of Theorem~\ref{ppa-cat0-strong} for totally bounded spaces.

\begin{theorem}\label{main-quant-thm}
Let $X$ be a totally bounded CAT(0) space with modulus $\alpha$ and $f : X \to (-\infty,\infty]$ be a 
convex lsc proper function with $F:=Argmin(f)\ne\emptyset$. 
Let $x\in X$  and $b>0$ be such that $d(x,p)\leq b$ for some minimizer $p$ of $f$.

Assume that  $(\gamma_n)$  is a sequence in $(0,\infty)$ such that $\sum_{n=0}^\infty \gamma_n = \infty$ with 
rate of divergence $\theta$ and $L:\mathbb{N}\to \mathbb{N}$  is a nondecreasing mapping satisfying $\displaystyle L(k)\geq \max_{0 \leq i \leq k} \gamma _i$ for all $k\in\mathbb{N}$.

Let  $(x_n)$ be the proximal point algorithm starting with $x$.
Define $\Psi_{b,\theta,L,\alpha}$ and $\Omega_{b,\theta,L,\alpha}$ as in Table \ref{tabel-1}. 
Then
\begin{enumerate}[(i)] 
\item $\Psi_{b,\theta,L,\alpha}$ is a rate of metastability for  ${(x_n)}$.
\item for all $k \in \mathbb{N}$ and all $g: \mathbb{N} \to \mathbb{N}$ there exists $N\leq \,\Omega_{b,\theta,L,\alpha}(k,g)$ such that
\[\!\!\!\!\!\!\!\!\!\!\!\!\!\forall i, j\!\in \![N,N+g(N)]\!\left(\!d(x_i,x_j) \leq \frac1{k+1}\!\right) \text{~and~}\, 
\forall m\geq N\!\left(\!x_m \leq \min(f) + \frac1{k+1}\!\right)\!.\]
\end{enumerate}
\end{theorem}
\begin{proof}
In the following claims, we show that the necessary moduli can be computed in our setting. We refer to \cite{KohLeuNic18} for the terminology. \\[2mm]
\noindent {\bf Claim 1:} $(x_n)$ is uniformly Fej\'er monotone with respect to $F$ with modulus $\chi_L$ defined in Table~\ref{tabel-1}, that is:
for all $n,m,r \in \mathbb{N}$, 
\begin{center}
 $q\in Argmin_{\chi_L(n,m,r)}(f)$ \quad implies\quad $d(x_{n+l},q) < d(x_n,q) + \frac1{r+1}$ for all $l \leq m$.
\end{center}
{\bf Proof of claim:} Let $n,m,r \in \mathbb{N}$, $q$ as in the hypothesis and $l \leq m$. 
By \eqref{dxnl-p}, we have that
\begin{equation}\label{fejer-eq1}
d(x_{n+l},q) \leq d(x_n, q) + \sum_{i=n}^{n+l-1} d(q, J_{\gamma_i} q)\leq d(x_n, q) +
\sum_{i=n}^{n+m-1} d(q, J_{\gamma_i}q).
\end{equation}
Let $i \in \{n,...,n+m-1\}$. We get that
\begin{align*} 
\gamma_i f(q) \leq &\, \gamma_i f(y) + \frac{\gamma_i}{\chi_L(n,m,r)+1} \quad \text{since~} q \in Argmin_{\chi_L(n,m,r)}(f)\\
\leq &\, \gamma_i f(y) + \frac{L(n+m-1)}{\chi_L(n,m,r)+1} \quad \text{by the definition of~} L\\
= &\, \gamma_i f(y) + \frac{1}{2(m(r+1)+1)^2} \quad \text{by the definition of~}\chi_L. 
 \end{align*}
Thus, $q$ is a $\psi(m(r+1))$-approximate fixed point of $\gamma_if$, where $\psi$ is defined by \eqref{appmin-def}.
We can apply Proposition~\ref{l-l} for $\gamma_i$ to  obtain that $q$ is an $m(r+1)$-approximate 
fixed point of $J_{\gamma_i}$, that is
$$d(q, J_{\gamma_i}q) \leq \frac 1{m(r+1)+1}.$$

It follows that 
$$\sum_{i=n}^{n+m-1} d(q, J_{\gamma_i}q) \leq \frac{m}{m(r+1)+1} < \frac1{r+1}.$$
By \eqref{fejer-eq1}, the claim follows. \hfill $\blacksquare$\\[1mm]

\noindent {\bf Claim 2:} $(x_n)$ is asymptotically regular with respect to $F$, with a rate of asymptotic regularity
$\beta_{b, \theta}$ defined by \eqref{rate-fxn}, that is: 
\[\forall k\in\mathbb{N}\, \forall n\geq \beta_{b, \theta}(k)\,\,\left(x_n\in Argmin_k(f)\right).\]
{\bf Proof of claim:} Let $k\in\mathbb{N}$. Applying Proposition~\ref{prop-limfxn-minf}, 
we get that for all $n\geq \beta_{b, \theta}(k)$, 
\[f(x_n)\leq \min(f)+\frac1{k+1}\leq f(y)+\frac1{k+1} \quad \text{for all~} y\in X. \tag*{$\blacksquare$}\]

\noindent {\bf Claim 3:}  $(x_n)$ has $F$-approximate points, with $\beta_{b, \theta}$ being an approximate $F$-point bound, that is: 
\[\forall k\in\mathbb{N}\, \exists N\leq \beta_{b, \theta}(k)\,\,\left(x_N\in Argmin_k(f)\right).\]
{\bf Proof of claim:} It is an obvious consequence of the second claim. \hfill $\blacksquare$\\[2mm]
Now we can finish the proof of the theorem.
\begin{enumerate}[(i)] 
\item We apply \cite[Theorem 5.1]{KohLeuNic18}, based on the first and the third claims. Using the notations 
from \cite{KohLeuNic18}, $\alpha_G:=id_{\mathbb{R}_+}, \beta_H:=id_{\mathbb{R}_+}$, $\gamma:=\alpha$, hence $P=\alpha(4k+3)$,
and, furthermore,
\begin{align*} 
\chi_{g,L}(n,r)=&\, \chi_L(n,g(n),r)= 2L(n+g(n)-1)(g(n)(r+1)+1)^2-1\\
\chi_{g,L}^M(n,r)=&\, \max_{i \leq n} \chi_{g,L}(i,r)=\max_{i \leq n} 2L(i+g(i)-1)(g(i)(r+1)+1)^2-1.
 \end{align*}
\item As pointed out in \cite[p. 23]{KohLeuNic18} and detailed in \cite{KohLeuNic18-arxiv} (the preprint version 
of \cite{KohLeuNic18}), we can get (ii) as an application of \cite[Theorem 5.8.(ii) and Corollary 5.9]{KohLeuNic18-arxiv}, 
based on the second claim and the following facts:
\begin{enumerate}
\item the rate of asymptotic regularity $\beta_{b,\theta}$ satisfies $\beta_{b,\theta}^M = \beta_{b,\theta}$.
\item $L^M=L$ (as $L$ is nondecreasing), hence, using the notation from 
\cite[Corollary 5.9]{KohLeuNic18-arxiv}, we have that $\chi_L^M=\chi_L$. 
\end{enumerate}
\end{enumerate}
\end{proof}

\begin{table}[ht!]
\begin{center}
\scalebox{0.95}{\begin{tabular}{ | l |}
\hline\\[1mm]
$\Psi_{b,\theta,L,\alpha}:\mathbb{N}\times \mathbb{N}^\mathbb{N}\to \mathbb{N}$, \quad $(\Psi_0)_{b,\theta,L}:\mathbb{N}\times\mathbb{N}\times \mathbb{N}^\mathbb{N}\to \mathbb{N}$\\[2mm]
$\Psi_{b,\theta,L,\alpha}(k,g)=(\Psi_0)_{b,\theta,L}(\alpha(4k+3),k,g)$\\[2mm]
$(\Psi_0)_{b,\theta,L}(0,k,g)=0$ \\[2mm]
$(\Psi_0)_{b,\theta,L}(n+1,k,g)=
\beta_{b,\theta}\big(\chi_{g,L}^M((\Psi_0)_{b,\theta,L}(n,k,g),4k+3)\big)$\\[2mm]
$\chi_L:\mathbb{N}\times\mathbb{N}\times\mathbb{N}\to \mathbb{N}, \quad \chi_L(n,m,r)= 2L(n+m-1)(m(r+1)+1)^2-1$\\[2mm]
$\chi_{g,L}^M:\mathbb{N}\times\mathbb{N} \to \mathbb{N}, \quad \chi_{g,L}^M(n,r)=\max_{i \leq n} 2L(i+g(i)-1)(g(i)(r+1)+1)^2-1$\\[5mm]
$\Omega_{b,\theta,L,\alpha} : \mathbb{N}\times \mathbb{N}^\mathbb{N}\to \mathbb{N}$, \quad $\Omega_{b,\theta,L,\alpha}(k,g)= 
\Psi^*_{b,\theta,L,\alpha}(k,g_{\beta_{b,\theta}(k)}) + \beta_{b,\theta}(k)$\\[2mm]
$\Psi^*_{b,\theta,L,\alpha}:\mathbb{N}\times \mathbb{N}^\mathbb{N}\to \mathbb{N}$, \quad 
$\Psi^*_{b,\theta,L,\alpha}(k,g)=\Psi_{b,\theta,L,\alpha^M}(k,g)$\\[2mm]
$g_l:\mathbb{N}\to\mathbb{N}, \quad g_l(n)= g^M(n+l)+l$\\[2mm]
with $\beta_{b,\theta}$ given by \eqref{rate-fxn}\\[1mm]
\hline 
\end{tabular}}\\[1mm]
\caption{Definitions of $\Psi_{b,\theta,L,\alpha}$ and $\Omega_{b,\theta,L,\alpha}$} \label{tabel-1}
\end{center}
\end{table}


\begin{thebibliography}{99}
\bibitem{Jos95}
J. Jost,
Convex functionals and generalized harmonic maps into spaces of non positive curvature,
Comment. Math. Helv. 70 (1995), 659--673.

\bibitem{Mar70}
B. Martinet,
R\'egularisation d'in\'equations variationnelles par approximations successives,
Rev. Fran\c caise Informat. Recherche Op\'erationnelle 4 (1970), 154--158.

\bibitem{Roc76}
R. T. Rockafellar,
Monotone operators and the proximal point algorithm,
SIAM J. Control Optim. 14 (1976), 877--898.

\bibitem{BreLio78}
H. Br\'ezis, P. Lions,
Produits infinis de r\'esolvantes, 
Israel J. Math. 29 (1978), 329--345.

\bibitem{BauCom17}
H. Bauschke, P. Combettes,
{\it Convex Analysis and Monotone Operator Theory in Hilbert Spaces}, Second Edition,
Springer, 2017.

\bibitem{Bac13}
M. Ba\v{c}\'ak,
The proximal point algorithm in metric spaces,
Israel J. Math. 194 (2013), 689--701.

\bibitem{Gul91}
O. G\"uler,
On the convergence of the proximal point algorithm for convex minimization,
SIAM J. Control Optim., 29 (1991), 403--419.

\bibitem{Koh08}
U. Kohlenbach,
{\it Applied proof theory: Proof interpretations and their use in mathematics},
Springer Monographs in Mathematics, Springer, 2008.

\bibitem{Koh17}
U. Kohlenbach, Recent progress in proof mining in nonlinear analysis,
IFCoLog Journal of Logics and its Applications 10 (2017), 3357--3406.


\bibitem{Koh18}
U. Kohlenbach, Proof-theoretic methods in nonlinear analysis, draft, 2017;
to appear in Proceedings of the ICM2018. 

\bibitem{Koh90}
U. Kohlenbach,
{\it Theorie der majorisierbaren und stetigen Funktionale und ihre Anwendung bei der Extraktion von Schranken aus inkonstruktiven Beweisen: Effektive Eindeutigkeitsmodule bei besten Approximationen aus ineffektiven Beweisen},
PhD Thesis, Frankfurt am Main, 1990.

\bibitem{Bri09}
E. M. Briseid,
Logical aspects of rates of convergence in metric spaces,
J. Symbolic Logic 74 (2009), 1401--1428.

\bibitem{LeuNicSip17}
L. Leu\c stean, A. Nicolae, A. Sipo\c s,
An abstract proximal point algorithm, 
arXiv:1711.09455 [math.OC], 2017.

\bibitem{Tao07} 
T. Tao,
Soft analysis, hard analysis, and the finite convergence principle,
Essay posted May 23, 2007, appeared in: 
T. Tao,
{\it Structure and Randomness: Pages from Year One of a Mathematical Blog},
Amer. Math. Soc., 2008.

\bibitem{Tao08} 
T. Tao,
Norm convergence of multiple ergodic averages for commuting transformations,
Ergodic Theory Dynam. Systems 28 (2008), 657--688.

\bibitem{KohLeuNic18}
U. Kohlenbach, L. Leu\c stean, A. Nicolae,
Quantitative results on Fej\'er monotone sequences, 
Commun. Contemp. Math. 20 (2018), 1750015 [42 pages].

\bibitem{KohLopNic17}
U. Kohlenbach, G. L\'opez-Acedo, A. Nicolae,
Quantitative asymptotic regularity for the composition of two mappings, 
Optimization 66 (2017), 1291--1299. 

\bibitem{LeuRadSip16}
L. Leu\c stean, V. Radu, A. Sipo\c s,
Quantitative results on the Ishikawa iteration of Lipschitz pseudo-contractions, 
J. Nonlinear Convex Anal. 17 (2016), 2277--2292.

\bibitem{LeuSip17}
L. Leu\c stean, A. Sipo\c s,
An application of proof mining to the proximal point algorithm in CAT(0) spaces, 
arXiv:1707.09169 [math.OC], 2017; to appear in: A. Bellow, C. Calude, T. Zamfirescu (eds.), 
{\it Mathematics Almost Everywhere. In Memory of Solomon Marcus}, World Scientific.

\bibitem{Ger08} 
P. Gerhardy, 
Proof mining in topological dynamics, 
Notre Dame J. Form. Log. 49 (2008), 431--446.





\bibitem{KohLeuNic18-arxiv}
U. Kohlenbach, L. Leu\c stean, A. Nicolae,
Quantitative results on Fej\'er monotone sequences, 
preprint version, arXiv:1412.5563 [math.LO], 2015.


\end{thebibliography}
\end{document}